\documentclass[]{aspm}

\articleinfo{The 50th Anniversary of Gr\"obner Bases}{75}{2017}

\usepackage{geometry}
\usepackage{verbatim}
\usepackage{amssymb}
\usepackage{amsbsy}
\usepackage{amscd}
\usepackage{amsmath}
\usepackage{amsthm}
\usepackage[mathscr]{eucal}

\newtheorem{theorem}{Theorem}\numberwithin{theorem}{section}
\newtheorem{corollary}[theorem]{Corollary}
\newtheorem{lemma}[theorem]{Lemma}

\newtheorem{definition}[theorem]{Definition}

\newtheorem{problem}[theorem]{Problem}
\newtheorem{remark}[theorem]{Remark}

 \usepackage{color,url, listings}

\def\ve#1{\mathchoice{\mbox{\boldmath$\displaystyle\bf#1$}}
{\mbox{\boldmath$\textstyle\bf#1$}}
{\mbox{\boldmath$\scriptstyle\bf#1$}}
{\mbox{\boldmath$\scriptscriptstyle\bf#1$}}}

\newcommand{\Z}{{\mathbb Z}}

\newcommand{\R}{\mathbb R}

\newcommand{\F}{\mathcal F}

\renewcommand{\H} {{\mathcal H}}

\newcommand{\Qsat}{\ensuremath{Q_{\operatorname{sat}}}}

\def\a{{\mathbf a}}
\def\b{{\mathbf b}}

\def\f{{\mathbf f}}
\def\h{{\mathbf h}}

\def\x{{\mathbf x}}
\def\v{{\mathbf v}}
\def\z{{\mathbf z}}

\newcommand{\rem}[1]{} 
\def\ve#1{\mathchoice{\mbox{\boldmath$\displaystyle\bf#1$}}
{\mbox{\boldmath$\textstyle\bf#1$}}
{\mbox{\boldmath$\scriptstyle\bf#1$}}
{\mbox{\boldmath$\scriptscriptstyle\bf#1$}}}

\definecolor{orange}{RGB}{255,127,0}

\title[Semigroups]{Semigroups -- A Computational Approach}

\author[F.~Kohl, Y.~Li,  J.~Rauh, and R.~Yoshida]{Florian Kohl \and Yanxi Li \and  Johannes Rauh \and Ruriko Yoshida}

\address{FB Mathematik and Informatik
Institut fuer Mathematik, K\"onigin-Luise Str. 24-26,14195 Berlin, Germany;\\Department of Mathematics and Statistics,
York University, 4700 Keele St.,
Toronto, ON, M3J 1P3;\\
Statistics Department,
University of Kentucky,
725 Rose Street,
Lexington KY 40536-0082;\\
Department of Operations Research,
1411 Cunningham Road,
Naval Postgraduate School,
Monterey, CA 93943-5219.\\
}

\email{fkohl@zedat.fu-berlin.de,\\jarauh@yorku.ca,\\yli282@g.uky.edu, \\ryoshida@np.sedu}

\rcvdate{}
\rvsdate{}

\subjclass[2010]{11P21, 52B11, 90C08, 97K70}

\keywords{Feasibility Problem, lattice points, polytopes, software}

\begin{document}

\begin{abstract}
The question whether there exists an integral solution to the system of linear equations with non-negativity constraints, $A\x = \b, \, \x \ge 0$, where $A \in \Z^{m\times n}$ and ${\mathbf b} \in \Z^m$, finds its applications in many areas such as operations research, number theory, combinatorics, and statistics. In order to solve this problem, we have to understand the semigroup generated by the columns of the matrix $A$ and the structure of the ``holes'' which
are the difference between the semigroup and its saturation. In this paper, we discuss the implementation of an algorithm by Hemmecke, Takemura, and Yoshida that computes the set of holes of a semigroup and we discuss applications to problems in combinatorics.  Moreover, we compute the set of holes for the common diagonal effect model and we show that the $n^\text{th}$ linear ordering polytope has the integer-decomposition property for $n\leq 7$. The software is available at
 \url{http://ehrhart.math.fu-berlin.de/People/fkohl/HASE/}.
\end{abstract}

\maketitle

\section{Introduction}
\label{sec:intro}\label{subsec:semigroup}

Consider the system of linear equations and inequalities
\begin{equation}\label{eq:feas}
A {\mathbf x} = {\mathbf b}, \, \, \, {\mathbf x} \geq 0, 
\end{equation}
where $A \in \Z^{m\times n}$ and ${\mathbf b} \in \Z^m$.  
 Suppose that the solution set over the real numbers $\{\mathbf x \in \R^n: A {\mathbf x} = {\mathbf b}, \, {\mathbf x} \geq 0\}$ is not empty.
\begin{problem}\label{prob:feas}
Decide whether there exists an integral solution to the system~\eqref{eq:feas} or not.
\end{problem}
Problem~\ref{prob:feas} is called the \emph{integer feasibility problem}.  To decide whether a system of equations is feasible is the first step in integer programming, where the goal is to find an ``optimal'' solution.  Therefore, problem~\ref{prob:feas} can be solved computationally using a linear programming system that can handle integer constraints, such as \textbf{lp\_solve}~\cite{lpsolve}.  However, this computational approach does not work if one wants to study a family of integer feasibility problems at the same time.  Here, we consider the following problem:
\begin{problem}\label{prob:feas2}
For fixed $A\in\Z^{m\times n}$, decide for which $\b\in\Z^{m}$ there exists an integral solution to the system~\eqref{eq:feas}.
\end{problem}

Problems~\ref{prob:feas} and~\ref{prob:feas2}
are of fundamental importance in many areas such as operations research, number theory, combinatorics, and statistics (see \cite{takemura2008} and references within). 
For instance, the \emph{Frobenius problem} is a simple-sounding yet wide-open integer feasibility problem, see e.g. \cite{Ramirez} for an overview. For coprime integers $a_1$, $a_2$,$\dots$, $a_n >0$, the Frobenius problem asks to find the biggest positive integer that cannot be expressed as a non-negative linear combination of the $a_i$'s with integral coefficients. Even for $n=4$, this is an active area of research. 

Feasibility can be described in terms of the \emph{semigroup}
\begin{equation}
\label{eq:semigroupQ}
Q = Q(A)= \left\{\ve a_1x_1 + \cdots + \ve a_n x_n \mid x_1, \dots, x_n \in   \Z_{\geq 0} \right\}
\end{equation}
generated by the column vectors $\ve a_1,\ldots,\ve a_n$ of~$A$.  Here, $\Z_{\geq 0}$ denotes the set of non-negative integers,
i.e., $\Z_{\geq 0} := \{0,1,2,\dots \}$.  Moreover, we need the \emph{cone}
\begin{equation*}
K = K(A) = \left\{\ve a_1x_1 + \cdots + \ve a_n x_n \mid x_1, \dots, x_n \in   \R_{\geq 0} \right\}
\end{equation*}
generated by the columns of~$A$, where $\R_{\geq 0} := [0,\infty)$.  Throughout this paper, we assume that all cones are
\emph{pointed}, i.e., that they do not contain lines: if $\mathbf{v}\in K\setminus\{0\}$, then $-\mathbf{v}\notin K$.
Finally, we need the \emph{lattice}
\begin{equation*}
  L = L(A) = \left\{\ve a_1x_1 + \cdots + \ve a_n x_n \mid x_1, \dots, x_n \in   \Z \right\}
\end{equation*}
generated by the columns of~$A$.  In this paper, we assume $L(A) =\Z^n$.

By definition, an integral solution to the system~\eqref{eq:feas} exists if and only if ${\mathbf b}\in Q$.
In general, it is difficult to check whether a given vector belongs to~$Q$.  However, it is much easier to check whether a given vector belongs to $L$ or to~$K$: To check whether $\mathbf{b}\in L$ is a problem of linear algebra (over the integers), and to check whether $\mathbf{b}\in K$ one can compute the inequality description of $K$ and check whether $\mathbf{b}$ satisfies all linear inequalities. Admittedly, computing the inequalities can be a difficult problem in itself, but usually it is still easier than the integer feasibility problem. Therefore, it makes sense to compare $Q$ to the larger semigroup $\Qsat = K \cap L$, which is called the
{\em saturation} of $Q$.  Clearly, $Q \subset \Qsat$, and
we call $Q$ {\em saturated} (or \emph{normal}) if $Q = \Qsat$.
We define the set of {\em holes $H$} of the semigroup $Q$ to be $H:=\Qsat \setminus Q$.

If $\b \in H \subset \Qsat$, then the system
\[
A\x = \b, \x \geq 0, 
\]
has a solution $\x\in\R^{n}$ over the \emph{reals}, but no solution $\x\in\Z_{\geq 0}^{n}$ over the integers.

In general, the set $H$ may be infinite, but it is possible to write $H$ as a finite union of finitely generated (affine) monoids. The first step is to compute the fundamental holes, where we say that a hole $\mathbf h\in H$ is \emph{fundamental} if there is no other hole $\mathbf h'\in H$ such that $\mathbf h-\mathbf h'\in Q$.  In contrast to $H$, the set
$F\subseteq H$ of fundamental holes is always finite, as it is
contained in the bounded set
\[
P : = \left \{\sum\nolimits_{i=1}^{n}\lambda_i \a_i \mid 0\leq \lambda_1,\dots,\lambda_n<1   \right \},
\] as shown in \cite{takemura2008}.
A finite algorithm to compute~$F$ is due to~\cite{raymond09}.

Once $F$ is known, it is necessary to compute an explicit representation of the holes in~$\f + Q$.  Hemmecke et.~al. \cite{raymond09} showed how the set of holes in~$\f + Q$ can be expressed as a finite union of finitely generated monoids using ideas from commutative algebra.  Together with an algorithm to compute the fundamental holes, this gives a \emph{finite} algorithm to compute an \emph{explicit} representation of $H$, even for an infinite set $H$.

As shown by \cite{Barvinok2003,takemura2008}, computing the set of holes is polynomial in time in the input size of $A$ if we fix the number of variables $m$ and~$n$ (see the definition of input size in~\cite{bar}).
Once we compute $Q$ for a particular matrix~$A$, we do not
have to compute it again as it does not depend on~$\b$.

In this paper, we have implemented the algorithm introduced in \cite{raymond09} and we have
applied our software to problems in combinatorics and statistics.  We named the software HASE (\textbf Holes in \textbf Affine \textbf{SE}migroups).
It is available at \url{http://ehrhart.math.fu-berlin.de/People/fkohl/HASE/}.  The homepage also contains the input files that are needed to reproduce the examples that are discussed in this paper.

This paper is organized as follows:  In Section \ref{ALG}, we outline
the algorithm. The performance of the algorithm and possible ways to speed up the process are described in Section \ref{sec:performance}. In Section~\ref{CDEM}, we compute the set of holes for the {\em common diagonal effect model} \cite{takemura2008b}. In Section~\ref{experiment}, we show some
computational experiments concerning the integer-decomposition property of polytopes and concerning a lifting algorithm for Markov bases, see \cite{RauhSullivant15:Lifting_Markov_bases}. We end with a discussion and open problems.

\section{Computing holes}\label{ALG}

In this section, we briefly describe our software and the implementation. 
The two main steps of the algorithm of \cite{raymond09} are:
\begin{enumerate}
\item  Compute the set $F$ of fundamental holes.
\item For each of the finitely many $\f\in F$, compute an
explicit representation of the holes in $\f+Q$.
\end{enumerate}
Our software outsources step $1$ to \textbf{Normaliz}, see \cite{Normaliz}.
\textbf{Normaliz} is a computer program that computes the saturation (or \textbf{Normaliz}ation) of an affine semigroup.
Usually, the saturation is output in the form of a matrix~$A'$ such that the saturation $\Qsat(A) = Q(A')$ equals the semigroup generated by the columns of~$A'$.
Starting with version~3.0, \textbf{Normaliz} can also compute a second representation of~$\Qsat(A)$ by giving a minimal set $F'$ of \emph{``generators of $\Qsat(A)$ as a $Q(A)$-module.''}  Formally, this says that
\begin{equation*}
  \Qsat(A) = \bigcup_{\f\in F'}(\f + Q(A)).
\end{equation*}
It is not difficult to see that $\textbf{0}\in F'$ (since $Q\subseteq \Qsat$) and that $F:=F'\setminus\{\textbf{0}\}$ is the set of fundamental holes of~$Q(A)$.  For details how \textbf{Normaliz} computes the set~$F'$, we refer to the documentation of~\textbf{Normaliz}. As an illustration, Section~\ref{CDEM} contains a description of the (fundamental and non-fundamental) holes of the common diagonal effect models.

It remains to determine the holes in $\f +Q$ for every fundamental hole $\f \in F$. Every \textbf{non}-hole belongs to $(\f + Q) \cap Q$ and if $\z \in (\f + Q) \cap Q$, then also $\z + A\boldsymbol{\lambda} \in (\f + Q) \cap Q$ for any $\boldsymbol{\lambda} \in \Z^n_+$. Consider the ideal
\begin{equation}
\label{eq:nonholeideal}
I_{A,\f} := \left \langle \x^{\boldsymbol{\lambda}} \mid \boldsymbol{\lambda} \in \Z^n_{\geq 0},\,  \f+ A\boldsymbol{\lambda} \in (\f + Q) \cap Q\right\rangle,
\end{equation} 
where $\x^{\boldsymbol{\lambda}}:=\prod_{i=1}^{n}x_{i}^{\boldsymbol{\lambda_{i}}}$ is the monomial with exponent vector~$\boldsymbol\lambda$.
Then, $\f  + A\boldsymbol{\lambda}$ is not a hole if and only if  $\x^{\boldsymbol{\lambda}} \in I_{A,\f}$. So we need to find a description of the monomials not in~$I_{A,\f}$.  These monomials are called the \emph{standard monomials}. There are algorithms for finding the standard monomials, once a generating set for the ideal $I_{A,\f}$ is known. A generating set of the ideal $I_{A,\f}$ is described by the following lemma:
\begin{lemma}[\cite{raymond09}, Lemma 4.1]
Let $M$ be the set of $\leq$-minimal solutions $(\boldsymbol{\lambda}, \boldsymbol{\mu})\in \Z_{\geq 0}^{2n}$ to $\f + A\boldsymbol{\lambda} = A \boldsymbol{\mu}$, where the partial order $\leq$  is given by coordinatewise comparison. Then
\[
I_{A,\f} = \left \langle \x^{\boldsymbol{\lambda}} \mid  \exists \boldsymbol{\mu} \in \Z_{\geq 0}^n \text{ such that }(\boldsymbol{\lambda}, \boldsymbol{\mu}) \in M\right \rangle	.
\]
\end{lemma}
Therefore, we have to find minimal integral solutions to the above system of linear equations for every fundamental hole $\mathbf f$. For this task, we can again use \textbf{Normaliz}, or we can use the \textbf{zsolve} command of \textbf{4ti2}, see \cite{4ti2}.  Usually, \textbf{zsolve} runs faster, so it is the default choice of HASE.

Once we have a generating set for~$I_{A,\f}$, we can use a computer algebra software to find the standard monomials.  In
general, the set of standard monomials of a polynomial ideal can be infinite, but it has a finite representation in
terms of \emph{standard pairs}.  HASE relies on \textbf{Macaulay2}~\cite{M2}, which has the command \texttt{standardPairs}.
A standard pair is a pair that consists of a monomial $\x^{\lambda}$ and a set $\x^{\boldsymbol\mu_{1}},\dots,\x^{\boldsymbol\mu_{r}}$ of monomials.  Such a pair corresponds to the set of holes
\begin{equation*}
  \textstyle \f + A (\boldsymbol\lambda + \sum_{i=1}^{r} c_{i}\boldsymbol\mu_{i}),
  \quad
  c_{i}\in\Z_{\geq 0},
\end{equation*}
and the set of all such standard pairs gives all holes in~$\f + Q$.

\section{Performance of the algorithm}
\label{sec:performance}

As shown by \cite{Barvinok2003,takemura2008}, computing the set of holes $H$ for the semigroup $Q$ is polynomial in
time in the input size of $A$ if we fix the number of variables $m$ and~$n$ (see the definition of the input size in~\cite{bar}).  Still, computing $H$ is a difficult problem, and our algorithm may fail to terminate due to limited memory or time even for reasonably-sized examples.

In the examples we computed, we experienced the following problems:
\begin{itemize}
\item \textbf{Normaliz} may fail to compute the set $F$ of fundamental holes.
\item For one of the fundamental holes~$\f\in F$, \textbf{zsolve} or \textbf{Normaliz} may fail to find the $\le$-minimal solutions to
  $\f+A\boldsymbol\lambda=A\boldsymbol\mu$.
\item For one of the fundamental holes~$\f\in F$, \textbf{Macaulay2} may fail to compute the standard pairs.
\end{itemize}
In this list, a failure means that either we ran out of memory or we ran out of time.

If \textbf{Normaliz} fails, there is not much we can do.  We really need the fundamental holes, and if computing the fundamental holes overstrains our computational resources at hand, it is very probable that the problem is just too difficult.  The only thing we could do is to ask the developpers of \textbf{Normaliz}, who are always up for a challenge, for advice.

If one of the later steps fails, there is much more that we could do.  The translation of computing the holes of the
form $\f + Q$ for a fundamental hole~$\f$ into a problem of commutative algebra is not very direct, and there may be some room for improvements.  We discuss one trick that we implemented in Section~\ref{sec:trick} below.

There may be fourth problem: Namely, the set $F$ may be extremely large.  Thus, even if \textbf{Normaliz} computes $F$
within reasonable time and if \textbf{zsolve} and \textbf{Macaulay2} find the hole monoids reasonably fast for each
single hole, the total running time may be unacceptable.  However, at least in this case it is relatively easy to obtain
a good estimate for the total running time that would be needed, since in this case the cardinality of~$F$ is known, and
the running times of \textbf{zsolve} and \textbf{Macaulay2} per fundamental hole can be estimated by their performance
on the first few holes.

If $F$ is very large, a natural remedy  is to look for symmetries of the problem.  However, currently symmetries are not implemented in HASE.

\subsection{Speeding up \textbf{zsolve}}
\label{sec:trick}

Let $\f\in F$ be a fundamental hole.  As explained in Section~\ref{ALG}, we want to solve the linear system $\f+A\boldsymbol{\lambda}=A\boldsymbol{\mu}$.
This system can be simplified considerably if certain non-holes are known in advance.  The simplest case is to look at
the vectors $\f+\a_{i}$, where $\a_{i}$ is a column of~$A$.

Suppose that $\f+\a_{i}$ is not a hole.  Then $\f+\a_{i}=A\boldsymbol{\mu_{0}}$ for some~$\boldsymbol{\mu_0}$.  Thus, $\f+\a_{i}+A\boldsymbol{\lambda} =
A(\boldsymbol{\mu_{0}}+\boldsymbol{\lambda})$.  This shows that if $\f+\a_{i}$ is not a hole, then $\f+\a_{i}+Q$ contains no other holes. This implies that every minimal solution to $\f+A\boldsymbol{\lambda}=A\boldsymbol{\mu}$ has $\lambda_{i}=0$. Let $A'$ be the matrix $A$ with the $i^\text{th}$ column $\a_{i}$ dropped. Then, instead of solving $\f+A\boldsymbol{\lambda}=A\boldsymbol{\mu}$, we may just as well solve $\f+A'\boldsymbol{\lambda'}=A\boldsymbol{\mu}$. Observe that this leads to a linear system with one variable fewer. If we can identify many columns $\a_{i}$ that we can drop, we can speed
up the computation of the holes in $\f + Q$.

This idea is implemented in HASE and can be activated using the option \texttt{--trick}.  With this option, HASE does
the following instead of solving $\f+A\boldsymbol{\lambda}=A\boldsymbol{\mu}$:
\begin{enumerate}
\item For each column~$\a_{i}$ of $A$, check whether $\f+\a_{i}$ is a hole.
\item Let $A'$ be the matrix with columns those $\a_{i}$ for which $\f+\a_{i}$ is a hole.
\item Compute the minimal solutions to $\f+A'\boldsymbol{\lambda'}=A\boldsymbol{\mu}$ (using either \textbf{zsolve} or \textbf{Normaliz}).
\item Use \textbf{Macaulay2} to compute the standard pairs of the ideal $I_{A',\mathbf f}$.
\end{enumerate}

Step~1 is an integer feasibility problem.  HASE uses the open source (mixed-integer) linear programming system
\textbf{lp\_solve}~\cite{lpsolve} to solve this problem.  Usually, this is a relatively quick step (and if it is not, it is again an
indication that our original problem is too difficult).

In the last step, observe that the trick also leads to a smaller ideal $I_{A',\f}$ or, more precisely, an ideal in a
smaller ambient ring ($I_{A,\f}$ and $I_{A',\f}$ will in fact have the same generators).  This, however, should not lead
to a big speed-up, since the command \verb|standardPairs| in \textbf{Macaulay2} will usually realize when variables do not
appear in the generating set of an ideal.

\section{Common diagonal effect models}\label{CDEM}

In this section, we consider the common diagonal effect models (CDEM) 
introduced by \cite{Hara09a}. 
The results were obtained by computing small examples using HASE to build a conjecture.

Let $A\in \Z^{(2d+1)\times d^2}$ be the matrix that computes the row sums, column sums, and also the diagonal sum of a $d \times   d$ table. For instance, if $ d =3 $, we have
\begin{equation*}
A = \begin{pmatrix}
1 & 0 & 0 & 1 & 0 & 0 & 1 & 0 & 0\\
0 & 1 & 0 & 0 & 1 & 0 & 0 & 1 & 0\\
0 & 0 & 1 & 0 & 0 & 1 & 0 & 0 & 1\\
1 & 1 & 1 & 0 & 0 & 0 & 0 & 0 & 0\\
0 & 0 & 0 & 1 & 1 & 1 & 0 & 0 & 0\\
0 & 0 & 0 & 0 & 0 & 0 & 1 & 1 & 1\\
1 & 0 & 0 & 0 & 1 & 0 & 0 & 0 & 1\\
\end{pmatrix}.
\end{equation*}
The cone $K$ generated by the columns of $A$ lies in the hyperplane
\begin{equation}
  \label{eq:sum}
  \sum_{i=1}^{d} z_i = \sum_{i=d+1}^{2d} z_i,
\end{equation}
since this linear equality is satisfied by all columns of~$A$.
Our goal is to describe the Hilbert basis of the saturation $\Qsat$ of the semigroup~$Q$ generated by the columns of $A$. First, we define a set $\F$ that will later turn out to be the set of fundamental holes of $Q$.
\begin{definition}
  Let $\a_{ij}$ be the $((i-1)d + j)^{\text{th}}$ column of~$A$, and let
\begin{equation*}
  \h_{kl} := \frac {1}{2} \left( \a_{ll} + \a_{lk} + \a_{kl} + \a_{kk} \right).
\end{equation*} 
Finally, let $\F := \left \{ \h_{kl} : k,l\in[d], k<l \right\}$,
where $[d] := \{1 , 2 , \dots,d\}$.
\end{definition}
There are $\binom{d}{2}$ choices for $l$ and $k$. Since every choice yields a different vector, we get $\left|\F \right| = \binom{d}{2}$. The next lemma shows that $\F$ consists of holes.
\begin{lemma}
$\F \subset K \cap \Z^{2d+1}$, and $\F\subseteq \Qsat\setminus Q$. 
\end{lemma}
\begin{proof}

$\a_{ij}$ has a $1$ in the $j^{\text{th}}$ coordinate and in the $(d +i)^{\text{th}}$ coordinate. Moreover, if $i=j$, then there is a $1$ in the $(2d+1)^{\text{th}}$ coordinate. So every vector of the form 
$2\h_{kl} = \a_{ll} + \a_{lk} + \a_{kl} + \a_{kk}$
has a $2$ in the $l^{\text{th}}$, $k^{\text{th}}$, $(d + l )^{\text{th}}$, $(d + k)^{\text{th}}$ and in the $(2d +1)^{\text{st}}$ coordinate, and all other coordinates are~0. Thus, $\F \subset K \cap \Z^{2d+1} = \Qsat$.

It remains to show that $\F \cap Q = \emptyset$.  So suppose we have an non-negative integral linear combination of the $\a_{ij}$'s that lies in~$\F$.  To get a $1$ in the $l^{\text{th}}$ coordinate, we need a generator $\a_{il}$ and to get a $1$ in the $k^{\text{th}}$ coordinate we need a generator~$\a_{i' k }$. Since there has to be a $1$ in the $(d+l)^{\text{th}}$ and $(d+k)^{\text{th}}$ coordinate, we see that $i$, $i' \in \{l,k \}$. Note that we cannot use a different generator to obtain a $1$ in these coordinates, since otherwise we would get another $1$ in the first $d$ coordinates. If there are more than $2$ generators, then either there is an entry bigger than $1$ or there are at least five $1'$s in the first $2d$ coordinates. 

If $i = i'$, then without loss of generality our linear combination is $\a_{ll} + \a_{lk}$, which
has a $2$ in the $(d+l)^{\text{th}}$ coordinate. If $ i \neq i'$, then
either we have $\a_{kl} + \a_{lk}$, which has a $0$ in the last entry, or
we have $\a_{ll} + \a_{kk}$ which has a $2$ in the last entry. Hence, we
have $\F \cap Q = \emptyset$.  To see that all elements
in $\F$ are indeed fundamental holes, one checks that each vector of the form $\h_{kl} - \a_{ij}$
is not in~$\Qsat$.
\end{proof}
We have now identified a set of fundamental holes. 
\cite[proof of Proposition 3.1]{takemura2008} have shown that the set of fundamental holes is contained in
\begin{equation*}
  P : = \left\{ \sum_{i,j \in [d]}
    \lambda_{ij} \a_{ij} \mid  0\leq \lambda_{ij} <1 \text{ for } i, j \in [d] \right\}.
\end{equation*}
To identify the fundamental holes, we can focus on $P$. Moreover, this proposition also implies that the (minimal) Hilbert basis is contained in the closure of $P$. 
The next theorem describes the (minimal) Hilbert basis for $\Qsat$.
\begin{theorem}
The minimal Hilbert basis for $\Qsat$ is given by
\[
\H : = \left\{ \a_{ij}\right\}_{i,j \in [d]} \cup \F.
\]
\end{theorem}
\begin{proof}
Let $\z=\left(z_1, z_2, \dots, z_{2d+1} \right) \in P \cap \Z^{2d+1}\setminus\{0\}$. %
Let $S := z_1 + z_2 + \dots + z_d = z_{d+1} + z_{d+2} + \dots +z_{2d}$. For every $i,j \in [d]$, we have $z_i = \sum_{j'=1}^d \lambda_{ij'}$ and $z_{d+j}\ge\lambda_{ij}$. This implies
\[
S - z_{d+i} = \sum_{j\in[d]\setminus\{i\}} z_{d+j} \geq \left\lceil \sum_{j\in[d]\setminus\{i\}} \lambda_{ij}\right\rceil \stackrel{\text{$z$ lattice point }}{=}  z_i ,
\] 
and hence
\begin{equation}
  \label{eq:Hall}
  z_i + z_{d+i}	
  \leq S. 
\end{equation}
We show that every non-negative integer vector $\z$ that satisfies~\eqref{eq:Hall} is a non-negative integer combination of~$\H$.  To do this, we show that if $\z\neq 0$, then there is an element $\a\in\H$ such that $\z-\a$ is non-negative and satisfies~\eqref{eq:Hall}.
Observe that subtracting $\a$ from $\z$ decreases the right hand side~$S$, so we need to make sure that the left hand side also decreases for those $i\in[d]$ for which~\eqref{eq:Hall} holds with equality.

First, suppose that $z_{2d+1} > 0$.

If there are two indices $l$ and $k$ where equality holds in~\eqref{eq:Hall}, then
\begin{equation*}
  z_{k} + z_{d+k} = S \ge z_{k} + z_{l}
  \quad\text{ and }\quad
  z_{l} + z_{d+l} = S \ge z_{d+k} + z_{d+l}.
\end{equation*}
It follows that $z_{d+k} = z_{l}$ and $z_{d+l} = z_{k}$.
Thus, $z_l$ and $z_k$ are the only nonzero entries among the first $d$ coordinates. 
It is easy to check that in this case, $\z - \frac12(\a_{ll} + \a_{lk} + \a_{kl} + \a_{kk})$ is non-negative and satisfies~\eqref{eq:Hall}.

If there is only one index $i$ for which equality holds in~\eqref{eq:Hall}, we can
check that $\z - \a_{ii}$ again is non-negative and satisfies~\eqref{eq:Hall}.
If there is no $i$ for which equality holds, we pick $i$ such that the pair of indices $(i, d+i)$ with $ z_i$, $ z_{d+i}\neq 0$ contains the biggest entry and subtract $\a_{ii}$.  Then $\z - \a_{ii}$ is non-negative and satisfies~\eqref{eq:Hall}.

It remains to discuss the case~$z_{2d+1} = 0$.  To express $\z$ as a non-negative linear integral combination of $\a_{ij}'$s where $i \neq j$, we translate the problem to a matching problem.  We have two labeled multi-sets $\bigcup_{i: z_{i}>0}\bigcup_{r=1}^{z_{i}}\{ i \}$ and $\bigcup_{j: z_{d+j}>0}\bigcup_{s=1}^{z_{d+j}}\{ j \}$, both of cardinality $S$.  Writing $\z$ as a non-negative integer combination of the $\a_{ij}$ corresponds to a matching between the two sets.  The matching has to be \emph{proper} in the sense that we only match elements $i,j$ with $i\neq j$.
For example, if $\z=(2,0,2,2,0,2,0)$, the two multi-sets are both equal to $\{1, 1, 3, 3\}$.
In this example, there is (up to symmetry) only one proper matching that matches 1 to 3 and 3 to~1, corresponding to the
identity $\z = \a_{1,3} +\a_{3,1} + \a_{1,3} + \a_{3,1}$.

It remains to show that there always exist such a proper matching.  This can either be seen directly by induction or by appealing to Hall's marriage theorem, noting that~\eqref{eq:Hall} always ensures that the marriage condition is satisfied.

This finishes the proof $\H$ is a Hilbert basis.  It is straightforward to check that $\H$ is indeed minimal.
\end{proof}
Knowing the Hilbert basis for $\Qsat$, we can completely describe the set of fundamental holes.
\begin{corollary}
$\F$ is the set of fundamental holes of~$Q$. 
\end{corollary}
\begin{proof}
  We have already seen that every element in $\F$ is a fundamental hole. 
  It only remains to show that there are no other fundamental holes.  Any fundamental hole is a non-negative integer
  combination of the Hilbert basis~$\H$.  Clearly, this combination cannot involve the columns~$\a_{ij}$ (otherwise the
  combination would not be fundamental).  Thus, it suffices to show that the sum of two holes in $\H$ is not a hole.
  This follows from the identity
    $\h_{ij} + \h_{kl} = \a_{kk} + \a_{ll} +\a_{ij}+\a_{ji}.$ 
\end{proof}

We have now seen that there are exactly $\binom{d}{2}$ \emph{fundamental} holes. However, this semigroup has \emph{infinitely} many holes:

\begin{theorem}
  \label{thm:CDEM-hole-monoids}
  The set of holes in the set $\h_{kl} + Q$ is the union of the two monoids
  \begin{subequations}
    \label{eq:CDEM-holes}
    \begin{equation}
      \label{eq:CDEM-holes-nd}
      \h_{kl} + \Z_{\geq 0} \a_{kk} + \Z_{\geq 0} \a_{kl} + \Z_{\geq 0} \a_{lk} + \Z_{\geq 0} \a_{ll}
    \end{equation}
    and
    \begin{equation}
      \label{eq:CDEM-holes-d}
      \h_{kl} + \sum\nolimits_{i=1}^{d}\Z_{\geq 0} \a_{ii}.
    \end{equation}
  \end{subequations}
\end{theorem}
\begin{proof}
  Fix $k,l\in[d]$, $k<l$.  If $i\neq j$ and $i\notin\{k,l\}$, then
  \begin{equation*}
    \h_{kl} + \a_{ij} = \a_{kk} + \a_{il} + \a_{lj}
  \end{equation*}
  assuming that $j\neq l$. If $j=l$, then we get 
\begin{equation*}
    \h_{kl} + \a_{ij} = \a_{ll} + \a_{kl} + \a_{ik}.
  \end{equation*} 
  Thus, if $i\neq j$ and $i\notin\{k,l\}$, then $\h_{kl} + \a_{ij} $ is not a hole. Similarly, if $j\notin\{k,l\}$, then $\h_{kl}+\a_{ij}$ is not a hole.  Thus, if $\h_{kl} + \sum_{r=1}^{s}\a_{i_{r}j_{r}}$ is a hole, then either $i_{r}=j_{r}$ or $\{i_{r},j_{r}\}=\{k,l\}$ for each~$r$.  We claim that either $i_{r}=j_{r}$ for all~$r$, or $\{i_{r},j_{r}\}=\{k,l\}$ for all~$r$.  This implies that each hole is as in the statement of the theorem. The claim follows from the computation
  \begin{equation*}
    \h_{kl} + \a_{ii} + \a_{kl} = \a_{kk} + \a_{ll} + \a_{il} + \a_{ki},
  \end{equation*}
  which is valid whenever $i\notin\{k,l\}$ and $k\neq l$.

  It remains to see that every integer vector in~\eqref{eq:CDEM-holes} is indeed a hole.  Let $\h$ be of the
  form~\eqref{eq:CDEM-holes-nd}, and suppose that $\h = \sum_{ij}\lambda_{ij}\a_{ij}$ with $\lambda_{ij}\ge0$.  Then
  $\lambda_{ij}\neq 0$ only for $\{i,j\}\subseteq\{k,l\}$, because $\h_{i}=0$ or $\h_{j}=0$ for $i,j\notin\{k,l\}$.
  The matrix $A$ restricted to the columns $\a_{ij}$ with $\{i,j\}\subseteq\{k,l\}$ equals
  \begin{equation*}
    \begin{pmatrix}
      1 & 0 & 1 & 0 \\
      0 & 1 & 0 & 1 \\
      1 & 1 & 0 & 0 \\
      0 & 0 & 1 & 1 \\
      1 & 0 & 0 & 1
    \end{pmatrix}
  \end{equation*}
  up to rows with only zeros.  Since this matrix has rank four, the representation of $\h$ as a linear combination is
  unique.  However, by assumption, $\h$ has a representation of the form $\h = \h_{kl} + \dots$ in which the
  coefficients are not integers (but half integers).  Thus, $\h$ is a hole.

  Finally, let $\h$ be of the form~\eqref{eq:CDEM-holes-d}.  Suppose that $\h = \sum_{ij}\lambda_{ij}\a_{ij}$ with
  $\lambda_{ij}\in\Z_{\geq 0}$, and let $S = \sum_{i=1}^{d}\h_{i} = \sum_{ij}\lambda_{ij}$.  Note that $\h_{2d+1} = S - 1$.
  Therefore, $\h$ is the sum of $S-1$ ``diagonal'' columns $\a_{ii}$ and one ``off-diagonal'' column $\a_{ij}$.  This is
  not possible, since, by assumption, the column sums and the row sums are the same.
\end{proof}

\begin{remark}
  Note the following two properties of the hole monoids:
  \begin{enumerate}
  \item The two monoids corresponding to a single hole $\h_{kl}$ are not disjoint.
  \item The hole monoids corresponding to two different fundamental holes $\h_{kl}$, $\h_{k'l'}$ are not disjoint: For example,
    \begin{equation*}
      \h_{12} + \a_{33} = \h_{23} + \a_{11}.
    \end{equation*}
  \end{enumerate}
\end{remark}

\section{Computational experiments}\label{experiment}
Semigroups play an important role in combinatorics, in discrete geometry, and in combinatorial commutative algebra.  The interplay between these areas is nicely exemplified by the theory of lattice polytopes. One can associate a semigroup to every lattice polytope. The Hilbert function of the corresponding graded semigroup ring turns out to be the Ehrhart function of the lattice polytope, see \cite[Section~12.1]{sturmfels2004} for more details about this connection and see \cite{beckrobins} for a nice introduction to Ehrhart theory. 

It is of particular interest to determine whether this semigroup has holes. For example, if the semigroup has holes, then there is no unimodular triangulation. Therefore, the algebraic structure is closely related to geometric properties. In Section~\ref{sec:IDP}, we briefly define what a (lattice) polytope is and how to construct the corresponding semigroup, and we present a computational result regarding the linear ordering polytope.

Holes of semigroups also play a role when computing Markov bases, as was recently shown by
\cite{RauhSullivant15:Lifting_Markov_bases}.  We give a brief example in Section~\ref{sec:lifting}.

\subsection{The Integer-Decomposition Property and Linear Order Polytopes}
\label{sec:IDP}
A \emph{polytope} $P\subset \R^d$ is the convex hull of finitely many vectors $\v_1$, $\v_2$, $\dots$, $\v_n \in \R^d$, and we write
\[
P = \operatorname{conv}\left(\v_1, \v_2, \dots, \v_n\right).
\]

The inclusion-minimal subset $V\subset \{\v_1, \v_2, \dots, \v_n \}$ such that $ P = \operatorname{conv}(V)$ is called the \emph{vertex set} of $P$ and an element of $V$ is called a \emph{vertex}. We say $P$ is a \emph{lattice polytope} if $V\subset \Z^d$, i.e., all coordinates of the vertices are integral.  The \emph{dimension} of a polytope $P$ is the dimension of its affine span, i.e., the dimension of the smallest affine subspace containing~$P$.  A $d$-dimensional polytope is sometimes called a \emph{$d$-polytope}.

A lattice $d$-polytope $P\in \R^d$ has the \emph{integer-decomposition property} (IDP) if for every integer $k>0$ and every integer point $\z \in kP\cap \Z^d$, there exist $\x_1, \x_2, \dots, \x_k \in P\cap \Z^d$ with
\begin{equation}
\label{eq:IDP}
\z = \x_1 + \x_2 + \dots + \x_k.
\end{equation}
Such a polytope is also called \emph{integrally closed} by other authors.

We can also express the integer-decomposition property in the language of semigroups. Let $(P,1):=\{(\x,1) \mid \x\in P\}\subset \R^{d+1}$ be the polytope embedded in $\R^{d+1}$ at height $1$. Moreover, let $K_{P} := \R_{\geq 0} (P,1)$ denote the (pointed) cone generated by the points in $(P,1)$ and let 
\[
Q_{P}:=\left\{\z\mid \z = k_1 \x_1 + \dots + k_n \x_n \text{, where } k_i \in \Z_{\geq 0}, \; \x_i \in (P,1)\cap\Z^{d+1} \right\}
\]
be the semigroup generated by the integer points in $(P,1)$. Then $P$ has the IDP if and only if the semigroup $Q_{P}$ is saturated. This means that we can check if a polytope has the integer-decomposition property by showing that the semigroup does not have any holes.

In a computational experiment, we examined whether the $7^\text{th}$ linear ordering polytope has the integer decomposition property. Sturmfels and Welker already showed that for $n\leq 6$, the $n^\text{th}$ linear ordering polytope satisfies the IDP, see \cite[Theorem 6.1]{SturmfelsWelker}.

For any permutation $\pi$ of $n$ elements, 
 we define
\[
v_{ij}(\pi) = \begin{cases}
1 		&\text{if }\pi(i) > \pi(j)\\
0		&\text{otherwise, }
 \end{cases}  
\]
where $1\leq i<j \leq n$. We follow the definition of \cite[Section 3.3]{Katthaen} 	and define the $n^\text{th}$ linear ordering polytope $P_n$ as the convex hull of the $n!$ vectors $\v(\pi) : = \left(v_{ij}(\pi)\right)_{1 \leq i<l\leq n} \in \R^{\binom{n}{2}}$. Note that the vertices are the only integer points of $P_n$. A python program that generates the matrix can be downloaded from
the HASE homepage.
After a bit more than a month of computation time on a linux machine
with 16 Intel processors (Intel(R) Xeon(R) CPU E5-2687W v2 at 3.40GHz), the program confirmed that $P_7$ also has the IDP.
\begin{theorem}
\label{LinOrderPoly}
The $n^{\text{th}}$ linear ordering polytope has the integer-decomposition property for $n\le7$.
\end{theorem} 
The question whether or not $P_n$ satisfies the IDP for all $n\in \Z_{\geq 1}$ is still open.

\subsection{Lifting Markov bases and Gr\"obner bases}
\label{sec:lifting}

Recently, Rauh and Sullivant \cite{RauhSullivant15:Lifting_Markov_bases} have proposed a new iterative algorithm to compute Markov bases and Gr\"obner bases of toric ideals in which a key step is to understand the holes of an associated semigroup. We do not explain this theory here, but we summarize two examples that arose in this context and that can now be reproduced using HASE.

The first example is from the computation of the Markov basis of the binary complete bipartite graph~$K_{3,N}$, as computed by~\cite{RauhSullivant14:Markov_basis_K3N}. The associated semigroup has two fundamental holes. Each fundamental hole has one associated monoid, generated by eight generators. The input file \texttt{K31codz.mat} for HASE can be downloaded from the HASE homepage.
Within a few seconds, HASE produces the following output:
{\footnotesize
\begin{lstlisting}
Normaliz found 2 fundamental holes.
Standard pairs of [1 1 1 1 1 1 1 1 1 1 1 1 0 1 1 0 1 0 0 1]:
   1: {x1, x2, x4, x7, x9, x10, x12, x15}
Standard pairs of [1 1 1 1 1 1 1 1 1 1 1 1 1 0 0 1 0 1 1 0]:
   1: {x0, x3, x5, x6, x8, x11, x13, x14}
\end{lstlisting}
}
The same method can be used to compute a Markov basis for the binary $3\times3$-grid. In this case, one needs to
understand the holes of a larger semigroup.  Again, the input file \texttt{3x3codz.mat} can be downloaded from the HASE homepage.

This problem turns out to be much more difficult
for HASE, and in fact, after waiting 24 hours for HASE to finish we became impatient and aborted the program, even when
the option \texttt{--trick} was activated.

Surprisingly, it turns out that the set of holes itself can be computed with some extra information: While the semigroup has 32 fundamental holes, there are only three symmetry classes.  The time that HASE spends on a
single hole varies greatly, even within a symmetry class. So all that is needed to finish the computation is to find
representatives of the three symmetry classes such that the hole monoid computations run through relatively quickly. The
current version of HASE cannot be used to run the algorithm on a subset of the fundamental holes (but it is not
difficult to do this manually by looking at HASE's source code). This shows once again how important it is to take
symmetry into account.

\section{Discussion and open problems}
There are many open problems concerning semigroups and holes of semigroups. In this section, we just want to briefly mention a non-respresentative selection of open problems. 

As mentioned in the beginning, the Frobenius problem is still open if there are more than two generators. There are several computational results, see e.g.~\cite{beckfrobenius}. It might be possible to use the structure of the holes, i.e. which hole is based on which fundamental hole, to say something about the Frobenius number. Alternatively, one could use a slightly modified version of HASE to compute the Frobenius number explicitly.


As briefly discussed in Section~\ref{sec:IDP}, holes in
semigroups coming from lattice polytopes are of particular interest as
they reflect geometric properties. Therefore, another application of
HASE is to describe the semigroup coming from a user specified lattice polytope.

\section{Acknowledgements}
The first author would like to thank Christian Haase for helpful discussions and for pointing his attention towards linear ordering polytopes. The first author was partially supported by a scholarship of the Berlin Mathematical School. We would also like to thank the \textbf{Normaliz} team for helpful discussions and suggestions.

\bibliographystyle{plain}
\bibliography{refs}






\end{document}